\documentclass[11pt,reqno]{amsart}

\usepackage{graphicx}
\usepackage{mathrsfs}
\usepackage{color}
\usepackage{comment}
\usepackage{amssymb}
\usepackage{amsmath}
\usepackage{mathabx}
\usepackage{amscd}
\usepackage{physics}
\usepackage{mathtools}
\usepackage{comment}
\usepackage{tikz, tikz-cd}
\usepackage[margin = 1in]{geometry} 
\usepackage[colorlinks=true, urlcolor=blue, citecolor=blue, linkcolor=blue]{hyperref}

\usepackage{enumitem}

\usepackage{microtype} 
\usepackage{lmodern}
\usepackage{csquotes}
\usepackage{nicefrac}

\calclayout

\let\oldtocsection=\tocsection
\let\oldtocsubsection=\tocsubsection
\let\oldtocsubsubsection=\tocsubsubsection

\renewcommand{\tocsection}[2]{\hspace{0em}\oldtocsection{#1}{#2}}
\renewcommand{\tocsubsection}[2]{\hspace{3em}\oldtocsubsection{#1}{#2}}
\renewcommand{\tocsubsubsection}[2]{\hspace{6em}\oldtocsubsubsection{#1}{#2}}

\numberwithin{equation}{section}
\newtheorem{theorema}{Theorem}

\theoremstyle{definition}
\newtheorem{theorem}{Theorem}[section]
\newtheorem{lemma}{Lemma}[section]
\newtheorem{prop}{Proposition}[section]
\newtheorem*{theorem*}{Theorem}
\newtheorem{corollary}{Corollary}[section]

\newtheorem{definition}{Definition}[section]

\newtheorem{conjecture}[]{Conjecture}

\newtheorem{remark}{Remark}[section]

\newcommand{\cons}{\nabla^{\Sigma}}
\newcommand{\conw}{\nabla^{W}}

\newcommand{\scal}{\mathrm{Sc}}
\newcommand{\scs}{\mathrm{Sc}_{\Sigma}}
\newcommand{\sw}{\mathrm{Sc}_W}
\newcommand{\sh}{\mathrm{Sc}_h}
\newcommand{\sgg}{\mathrm{Sc}_{\bar{g}}}
\newcommand{\sg}{\mathrm{Sc}_{g}}
\newcommand{\lscal}{\underline{\mathrm{Sc}}}
\newcommand{\mc}{\mathrm{H}_{\Sigma}}
\newcommand{\n}{\mathbf{n}}
\newcommand{\inner}[2]{\langle #1,#2  \rangle}

\newcommand{\dw}{D_W}
\newcommand{\daps}{D_W^{\mathrm{APS}}}
\newcommand{\dapsp}{D_{\mathrm{APS}}^{+}}
\newcommand{\dapsm}{D_{\mathrm{APS}}^{-}}

\newcommand{\ind}{\mathrm{Ind}}
\newcommand{\indaps}{\mathrm{Ind}_{\mathrm{APS}}}
\newcommand{\haps}{\mathcal{H}^-}
\newcommand{\kpt}{K3\backslash\{p\}}
\newcommand{\geu}{g_{\mathrm{euc}}}
\newcommand{\xs}{X\backslash S}
\newcommand{\QQ}{\mathbf{Q}}
\newcommand{\RR}{\mathbf{R}}

\newcommand{\HH}{\mathbf{H}}

\DeclareMathOperator{\ric}{Ric}
\DeclareMathOperator{\sff}{II}

\begin{document}

\title{Spin and scalar curvature on punctured four manifolds}
\author{Aditya Kumar}
\address{Department of Mathematics, University of Maryland, College
Park, MD 20742, USA}
\email{akumar65@umd.edu}
\author{Boyu Zhang}
\address{Department of Mathematics, University of Maryland, College
Park, MD 20742, USA}
\email{bzh@umd.edu}

\begin{abstract}
    We show that a closed spin four-manifold with nonzero signature, even after being punctured along a finite graph, does not admit a complete metric with uniformly positive scalar curvature. We also show that this implies Schoen's conjecture on $L^\infty$-metrics with non-negative scalar curvature for these manifolds. This is in contrast to higher dimensions. A key new observation, which could be of independent interest, is that Gromov's $\mu$-bubbles are a convenient class of boundaries for the APS boundary value problem for the Dirac operator. 
\end{abstract} 
\maketitle
\tableofcontents
\section{Introduction}\label{sec:intro}

It follows from Gromov's $h$-principle for open manifolds that every non-compact smooth manifold admits a Riemannian metric of positive scalar curvature (abbreviated to psc henceforth). Such a metric, however, is not necessarily complete. A substantially more difficult problem is to determine which non-compact smooth manifolds admit complete Riemannian metrics whose scalar curvature is bounded below by a uniform positive constant. The problem was considered in \cite{gl84}, and has since received quite a bit of attention. The case of non-compact $3$-manifolds was addressed in \cite{cwy10, bbm11,wang}, the case of non-compact manifolds with a torus connected summand in dimensions at most $7$ was addressed in \cite{luy24, cl24} and this was extended to aspherical summands recently in \cite{ccz23}. See also, \cite{cwy17, cwy20, cmm24,simone}.

It is observed by Gromov in \cite[\S 11.12 following Conjecture D']{geo18} that since a closed simply connected manifold $X^n$ after puncture contracts to an $(n-2)$ dimensional subpolyhedron, ``\textit{which, most probably implies that} $X\backslash\{x_0\}$ \textit{admits a complete metric with} $\scal >0$". In dimensions $n \geq 5$, this is indeed the case \cite{stolz,rw23}. This is also trivially true for $n=3$ as, by the resolution of the Poincar\'e conjecture, $\RR^3$ is the only possibility. For $n=4$, the classical test case is the $K3$ surface which does not admit a psc metric \cite{lich}. In \cite{cmm24}, the authors proved that there is some exotic punctured $K3$ surface that does not admit a complete uniformly psc metric. However, the standard case is unknown \cite[Remark 1.6]{cmm24}; it also appears in the K3 Problem List for Low Dimensional Topology \cite[Problem 4.128, Remark (6)]{k3}. %

We answer this question negatively. More generally, we show that in dimension four, the $\widehat A$-obstruction to psc in spin manifolds, survives puncturing along a \textit{finite graph}, i.e., a subcomplex of dimension at most one in a smooth triangulation of $X$.

\begin{theorema}\label{maintheoremB}
Let $X$ be a closed spin four-manifold with non-zero signature and let $S \subset X$ be a finite graph. Then, $\xs$ does not admit a complete Riemannian metric with $\scal \geq \kappa^2$.    
\end{theorema}
\begin{remark}
    We note that the conclusion for Theorem \ref{maintheoremB} holds more generally for any closed set $S \subset X$ with Hausdorff dimension strictly less than two\footnote{An interested reader can check this by noting that all one needs is a compact smooth domain $U$ containing $S$ and satisfying $H_2(U;\QQ)=0$. As long as $\dim S<2$, this can be done by perturbing the $2$-skeleton off $S$. Then the desired $U$ is just the complement of a small neighbourhood of the perturbed $2$-skeleton.} and this is sharp in the sense that the dimension of $S$ can not be equal two. Indeed, the complement of the $2$-skeleton of $X$ is the interior of a tubular neighbourhood of the dual $1$-skeleton, which is diffeomorphic to a connected sum of copies of $S^1 \times \RR^3$, which does admit such a metric. 
\end{remark}

The proof of Theorem \ref{maintheoremB} rests on the following vanishing result for the APS index for even dimensional compact spin manifolds whose boundary is a $\mu$-bubble, which could be of independent interest. 
\begin{theorema}\label{thm:apsintrovanish}
Let $(W^n,\Sigma,g)$ be an even dimensional compact spin manifold with boundary $\Sigma$ and $\sg>0$. If the boundary satisfies$$\lambda_1\bigg(-\Delta_{\Sigma} + \frac{1}{2}(\scs - \mc^2)\bigg)>0.$$ Then, $\indaps(W,g)=0$.      
\end{theorema}
This spectral condition is satisfied in particular if the boundary is a smooth $\mu$-bubble with the appropriate potential. It is in this sense that we say that $\mu$-bubbles are a convenient class of boundaries for the APS boundary value problem
\subsection{Metrics with codimension three singularities} An $L^\infty$-metric $g$ on a closed manifold $M$ is a measurable section of symmetric $2$-tensors such that, there is a smooth metric $g_0$ and a constant $\Lambda \geq 1$ such that $$\Lambda^{-1}g_0 \leq g \leq \Lambda g_0 \qquad \text{a.e. on }M.  $$ 
Our Theorem \ref{maintheoremB} can be applied to the following conjecture due to Schoen; cf. \cite[Conjecture 1.5]{lm19}.

\begin{conjecture}[Schoen]
\label{conj_Schoen}
    Suppose, $g$ is an $L^\infty$-metric on $M$ that is smooth away from a closed embedded submanifold $S$ of codimension at least $3$. If $M$ has non-positive Yamabe invariant and $\scal_g \geq 0$ on $M\backslash S$, then $g$ extends smoothly to $M$ as a Ricci-flat metric. 
\end{conjecture}
Here during the extension, one is allowed to change the smooth structure near $S$ \cite[Remark 1.4]{cfz25}. It was proven in dimension three in \cite{lm19}. In dimension four it is known for enlargeable four manifolds \cite{kaz24}. See also \cite{lt25,as26,wwx26}. It is known to be false in dimensions $n \geq 8$ \cite{cfz25}. 

The following observation provides an intermediate step to Conjecture \ref{conj_Schoen} in dimension four. Here the extension across $S$ is by a recent result of Antonelli-Székelyhidi \cite{as26}.

\begin{theorema}\label{thm:singularthm}
Let $X$ be a closed connected four-manifold and let $S \subset X$ be a finite graph such that $X \backslash S$ does not admit a complete metric with uniformly positive scalar curvature. Then any $L^\infty$-metric on $X$ with $\scal_g \geq 0$ on $X \backslash S$ has $\ric_g \equiv 0$ on $X \backslash S$ and it extends smoothly to a Ricci-flat metric on $X$ with a smooth structure that agrees with the original one on $X\backslash S$.
\end{theorema}
 Therefore, to establish Conjecture \ref{conj_Schoen} in dimension four, it is sufficient to show that if $X$ does not admit a psc metric then $\xs$ does not admit a complete uniformly psc metric. Combined with Theorem \ref{maintheoremB}, this resolves the conjecture for spin four manifolds with non-zero signature. For such manifolds there is also a classical rigidity consequence due to Lichnerowicz formula and Hitchin's theorem \cite[Theorem 1]{hit74} that any smooth metric with $\sg \geq 0$ is actually hyperkähler and the manifold is diffeomorphic to the $K3$ surface. Our result implies that this rigidity persists for metrics that have singularities along a finite graph.
\begin{theorema}\label{thm:rigidity}
    Let $X$ be a closed connected spin four manifold with non-zero signature and let $S \subset X$ be a finite graph. Any $L^\infty$-metric on $X$ that is smooth on $\xs$ with $\sg \geq 0$ extends to a hyperk\"ahler metric on $X$ for a smooth structure that agrees with the original one on $\xs$. With this smooth structure $X$ is diffeomorphic to the $K3$ surface. 
\end{theorema}

Theorem \ref{thm:rigidity} has intriguing connections to some of the outstanding problems in smooth four-manifold topology, which we discuss in the following remark.
\begin{remark}
    Suppose $X$ is an exotic $K3$ surface with an $L^\infty$-metric $g$, such that $\sg \geq 0$ on $X \backslash p$. Then, the smooth structure in Theorem \ref{thm:rigidity} is constructed on the completion of $(X\backslash p,g)$ cf. \cite[proof of Theorem 1.4]{as26} and in particular it agrees with the original one on $X\backslash p$, and with it, $X$ is diffeomorphic to the standard $K3$. So, there is a homeomorphism $f: X \to K3$ that is a diffeomorphism away from $p$. Now, take a smooth ball $B \subset K3$ around $f(p)$. Then, $D \coloneqq f^{-1}(B)$ is a topological four ball in $X$ whose boundary is a smooth three sphere as $f$ is diffeomorphic away from $p$. Capping $D$ off with a ball gives a homotopy four sphere $\Sigma$, such that $X \cong K3 \# \Sigma$, because by $f$, $X \backslash \mathrm{int} D \cong K3 \backslash \mathrm{int} B$. Since by the assumptions, $X$ is not diffeomorphic to $K3$, we know that $\Sigma$ is not diffeomorphic to the standard four-sphere. Therefore, if the 4-dimensional smooth Poincar\'e conjecture holds, then an exotic K3 surface $X$ cannot admit $L^\infty$ metrics that are smooth on $X\setminus p$ with $\sg\ge 0$. 
    
    Moreover, note that $f$ is bi-Lipschitz as $g$ is an $L^\infty$-metric, i.e., if an exotic K3 surface $X$  admits an $L^\infty$ metric that is smooth on $X\setminus p$ with $\sg\ge 0$, then $X$ must be bilipschitz homeomorphic to $X$. So far, there are no known examples of 4-manifolds (not necessarily K3) that are bi-Lipschitz homeomorphic but not diffeomorphic. 
\end{remark}

\subsection{Overview of the argument} In the overview, we shall assume that $S$ is just a point as the same argument works in general. One starts with a complete metric with $\scal \geq \kappa^2$ on $X \backslash \{p\}$. The natural approach is to first cut the end along a $\mu$-bubble $\Sigma$, giving a compact manifold $W$ containing all topology of $X$ in the interior and with boundary $\Sigma$. Then, one would try to construct a spin psc filling of $\Sigma$, giving a psc metric on $X$. This would contradict the fact that $\widehat A(X) \neq 0$. However, implementing this approach is challenging because, \textit{a priori} $\Sigma$ need not be an $S^3$, and constructing a psc filling for a psc three-manifold in general is a difficult question because one does not know all the possible psc three-manifolds that smoothly embed into $S^4$; cf. \cite[\S 5.2]{li2026}. 

Our first observation is that a psc filling is not necessary.  It is enough to construct a spin filling $C$ of $\Sigma$ with vanishing index (\S \ref{sec:filling}), as then by the index gluing formula $\widehat A(X) = \indaps(W,g)$. If the right side vanishes, we get a contradiction. Our second observation is that the right side vanishes when the hypersurface $\Sigma$ satisfies a positivity condition on the first eigenvalue of a modified stability operator(Definition \ref{def:modstab}); cf. Theorem \ref{thm:apsintrovanish} . In particular, for complete uniformly psc metrics, the $\mu$-bubble on the end satisfies this condition. This second step above is interesting because usually, for the vanishing of $\indaps(W,g)$ with $\sw \geq \kappa^2$, one needs the boundary $\Sigma$ be mean convex. We replace that with a weaker spectral condition (cf. Theorem \ref{thm:indexwvanish}) that is automatically satisfied by $\mu$-bubbles as long as the end is long enough. Also see Remark \ref{rem:defn}. 

We also note here that dimension four enters in two ways. One, we use the classification of closed psc three-manifolds. Second, the arguments establishing the vanishing of spin fillings of $\Sigma$ rely on Seiberg-Witten theory (cf. Proposition \ref{prop_Froyshov-Baraglia}). 

To prove Theorem \ref{thm:singularthm}, we follow an argument of \cite{lm19} to show that if $g$ is not Ricci-flat on $\xs$, then it can be perturbed to a metric with $\scal >0$. We then do a conformal deformation to obtain a complete uniformly psc metric on $\xs$, but the latter cannot exist hypothesis. The extension to a Ricci-flat metric uses a recent result of Antonelli and Székelyhidi \cite[Theorem 1.2]{as26}. 

\subsection{Stabilisation and $S^1$-stability and punctured $K3$ surface} The $S^1$-stability conjecture of Rosenberg \cite[Conjecture 1.24]{ros07} predicts that a closed manifold $M \times S^1$ admits a psc metric if and only if $M$ does. In dimension four, it is known to be false due to examples based on Seiberg-Witten theory \cite[Remark 1.25]{ros07}. Nonetheless, the first author and Sen showed in a previous work that in dimension four, the failure of $S^1$ stability can be seen by stabilisations by $\# (S^2\times S^2)$. Writing $H= S^2 \times S^2$, we showed that $M^4\times S^1$ admits a psc metric iff for some $k\geq 0$, $M\#^kH$ does.  
By Theorem \ref{maintheoremB}, we now know that $\kpt$ does not admit a complete uniformly psc metric. Interestingly, it is shown in \cite{rw23} that $\kpt \times T^2$ does admit a complete uniformly psc metric. In light of Theorem \ref{maintheoremB} and the preceding discussion, this suggests that both $\kpt \times S^1$ and $\kpt \#^{\infty} H$ admit a complete uniformly psc metric.\footnote{Note that a finite number of stabilisations can not work by Theorem \ref{maintheoremB} as $\widehat A(K3 \#^k H)=2$ for any finite $k$.} Indeed, in a new version of \cite{ks25} we confirm this expectation. Write $N \coloneqq \kpt$, then we prove that $N\#^{\infty} H$ as well as $N\times S^1$ admit complete uniformly psc metrics. The argument for $N\#^{\infty} H$ is quite short: since $K3\#\overline{K3} \approx_{diff} 22H$, a Mazur swindle argument gives the result. For $N \times S^1$, the arguments of \cite{ks25} applied to $N\#^{\infty} H$ allow one to conclude\footnote{Full details will be provided in a new version of \cite{ks25}.}.

We note that recently, Chen \cite{chen26} has given a different proof of non-existence of complete uniformly psc metric on $\kpt$ by extending the obstruction to psc from non-vanishing Seiberg-Witten invariant to four manifolds punctured at  finitely many points. 

\subsection{Organisation} Section \ref{sec:stability} defines the modified stability operator $\mathcal{L}_{\Sigma}$ and constructs $\mu$-bubbles with $\lambda_1(\mathcal{L}_{\Sigma}) > 0$. Section \ref{sec:bdrydirac} recalls basic notions of boundary value problems for the Dirac operator and establishes that the index of the compact core vanishes for boundaries with $\lambda_1(\mathcal{L}_{\Sigma}) > 0$. Section \ref{sec:filling} establishes the vanishing of the index for spin fillings of psc hypersurfaces in four maifolds without rational second homology. In Section \ref{sec:final} we combine the results to prove Theorem \ref{maintheoremB}. In Section \ref{sec:six} we prove Theorem \ref{thm:conf} and using that in Section \ref{sec:seven} we prove Theorems \ref{thm:singularthm} and \ref{thm:rigidity}.

\subsection*{Acknowledgements} A.K. owes tremendous intellectual debt to the several conversations with Balarka Sen over the years during which we thought about the this problem. Even though the ideas in this work can be traced to those discussions, he declined to be a coauthor. He also wishes to specially thank Demetre Kazaras for first pointing out that the results in earlier version of this paper also imply the Schoen conjecture; for explaining to us the details of the conformal blowup construction in \cite{kaz24,lm19}; and also for comments on earlier draft that helped improve the exposition. He also thanks Jeffrey Case for insightful comments on the material in \S 6.1; Gioacchino Antonelli for discussions on \cite{as26} as well as for helpful suggestions; Jonathan Rosenberg for several helpful discussions; Bernhard Hanke and Mike Miller Eismeier for helpful comments; Christos Mantoulidis, Shmuel Weinberger, and Guoliang Yu for their interest in this work. B.Z. thanks Chao Li for many valuable discussions over the years, as well as for their earlier collaborations on related topics. 
B.Z. is partially supported by NSF grants DMS-2540516, DMS-2405271, and a travel grant from the Simons Foundation.

\subsubsection*{AI statement} No AI tools were used for any ideas, computations, or writing in this paper. 

\section{Modified stability operator and $\mu$-bubbles} \label{sec:stability}

Let $\Sigma$ be a closed two-sided hypersurface in $(W,g)$ with unit normal $\n$, second fundamental form $\sff$, and mean curvature $\mc$. For our application, it is convenient to consider the following operator associated to $\Sigma$.

\begin{definition} \label{def:modstab}
    The modified stability operator of $\Sigma$ is \begin{equation}
        \mathcal{L}_{\Sigma} \coloneqq  -\Delta_{\Sigma} + \frac{1}{2}(\scs - \mc^2).
    \end{equation}
\end{definition}
\begin{remark}\label{rem:defn}
    We briefly explain the motivation for this definition. In the usual Schoen-Yau conformal descent argument, if the ambient manifold has $\scal>0$, and a hypersurface satisfies $\lambda_1(-\Delta_{\Sigma} + \scs/2)>0$, then by a comparison with the conformal Laplacian, this implies that $\Sigma$ admits a metric with $\scs >0$. Obviously, $\lambda_1(\mathcal{L}_\Sigma) > 0$ also implies that. However, we further observe that $\lambda_1(\mathcal{L}_\Sigma) > 0$ is also a sufficient condition for showing that the right-hand side of the boundary Lichnerowicz formula \eqref{eq:bdrylich} is non-positive for spinors satisfying the APS boundary condition. As a consequence, one gets that hypersurfaces satisfying  $\lambda_1(\mathcal{L}_\Sigma) > 0$, in particular Gromov's $\mu$-bubbles, are a natural class of boundaries for the APS boundary value problem; see $\S$\ref{subsec:vanish}. 
\end{remark}

\begin{prop}\label{prop:spectralgap}
Let $M$ be a closed manifold of dimension $3 \leq n\leq 7$, and let $S$ be a finite graph such that $M \backslash S$ has a complete metric $g$ with $\scal_g\geq\kappa^2$. Then, every neighbourhood $U$ of $S$ contains a smooth closed embedded hypersurface $\Sigma \subset U \backslash S$  separating $S$ from $\partial U$, such that $\lambda_1(\mathcal{L}_{\Sigma})> 0.$

\end{prop}
\begin{proof}
By completeness of $g$, the distance function $\mathrm{dist}_g(\cdot, \partial U) \to \infty$ at $S$. Let $\rho$ be the smoothened distance function on $U\backslash S$ such that $|\nabla \rho|\leq 1$. Now, fix $L > 2\pi \sqrt{n-1}/\kappa$. By completeness, choose $a$ large enough so that $V \coloneqq \{a \leq \rho \leq a+L\}$ is a compact band in $U \backslash S$. We take the prescribing function for the mean curvature to be $$h = \frac{2\pi(n-1)}{L}\cot (\pi \frac{\rho-a}{L}).$$ 

Then $h \to \pm \infty$ on the respective faces of $V$. The $\mu$-bubble existence theorem \cite[Proposition 12]{cl24}, applied to each component of $V$ meeting both faces gives a smooth closed embedded hypersurface $\Sigma \subset V$ with mean curvature $\mc = h$, separating $S$ from $\partial U$ as any path from $\partial U$ to $S$ crosses a component of $V$ from end to end. We shall denote by $C$ the closure of the union of the components of $M\backslash \Sigma$ that meet $S$, and by $W$ the closure of its complement. Since $\Sigma$ separates $S$ from $\partial U$, we have $S \subset C \subset \mathrm{int} U$, and $W \subset M\backslash S$, further, we shall also denote $\Sigma$ by $\partial W$. 

From the second variation formula, we have the following for every test function $f$ \cite[Lemma 14]{cl24} $$\int_{\Sigma} |\nabla f|^2+\frac{1}{2}(\scs-\mc^2)f^2 \geq \int_{\Sigma} (\scal_g/2 + |\sff|^2/2 + \n(h))f^2. $$
Since, $\dim \Sigma = n-1$, we have $|\sff|^2 \geq \mc^2/(n-1)=h^2/(n-1)$. Therefore, the integrand on the right-hand side can be estimated as
$$\frac{1}{2}\scal_g+\frac{1}{2}|\sff|^2+\n(h) \geq \frac{\kappa^2}{2}+\frac{h^2}{2(n-1)}-|\nabla h| \geq \frac{\kappa^2}{2} - \frac{2(n-1)\pi^2}{L^2},$$
here the last inequality used $|\nabla \rho| \leq 1$. Therefore, we have $$\lambda_1(\mathcal{L}_\Sigma) \geq  \frac{\kappa^2}{2} - \frac{2(n-1)\pi^2}{L^2}.$$
Since the metric is complete, $L$ can be arbitrarily large. In particular for any $L > 2\pi\sqrt{n-1}/\kappa$, we have $\lambda_1(\mathcal{L}_\Sigma)>0$. 
\end{proof}
\begin{corollary}
    The metric induced on $\Sigma$ is conformal to a psc metric. 
\end{corollary}
\begin{proof}
    Follows by comparison with the conformal Laplacian $L_\Sigma$ as $ 0< \lambda_1(\mathcal{L}_{\Sigma})< \lambda_1(L_{\Sigma})$.
\end{proof}

\section{The boundary Lichnerowicz identity and the APS index} \label{sec:bdrydirac}

In this section we will recall the boundary Lichnerowicz identity as well as the notion of the Atiyah-Patodi-Singer(APS) index. We will then use the modified stability operator to prove that the index vanishes on a domain with a psc metric whose boundary satisfies $\lambda_1(\mathcal{L}_{\Sigma})>0$.

\subsection{Notation and conventions} Let $(W^n,g)$ denote a compact Riemannian spin manifold with smooth boundary $\Sigma$. We will follow the setup of \cite[{\S 2}]{bar26}. We write $S_W$ for the complex spinor bundle, $\conw$ for the spin connection, and $\dw$ for the Dirac operator. A spinor $\psi$ is harmonic if $\dw \psi =0$, and $\phi = \psi_{\Sigma}$ denotes its restriction to the boundary.  $\mc$ denotes the mean curvature of $\Sigma$ with respect to the outward unit normal $\n$.\footnote{We normalise mean curvature as the sum of the principal curvatures. } 
If $n$ is odd, then $S_W|_{\Sigma}$ is canonically identified with the spinor bundle $S_{\Sigma}$ of $(\Sigma,g_{\Sigma})$; if $n$ is even then with $S_{\Sigma} \oplus S_{\Sigma}$, and the two summands are $S_W^{\pm}|_{\Sigma}$. If we denote the Dirac operator of $\Sigma$ by $D_{\Sigma}$, then the \textit{boundary Dirac operator} is 
\begin{equation}
    A \coloneqq D_{\Sigma} \qquad n\text{ is odd}, \qquad A\coloneqq \biggl( \begin{matrix} D_{\Sigma} & 0 \\ 0 & -D_{\Sigma} \end{matrix} \biggr) \qquad n \text{ is even}.
\end{equation}

It is a self-adjoint elliptic operator on the closed manifold $\Sigma$ with discrete real spectrum, it anti-commutes with $c(\n)$ \cite[Proposition 2.3]{bar96}, and it satisfies the Lichnerowicz formula \begin{equation*}
    A^2 = (\cons)^*\cons + \frac{1}{4}\scs,
\end{equation*}
where $\cons$ is the spin connection on $\Sigma$. The boundary Dirac operator $A$ is connected to the ambient geometry by the following relation along $\Sigma$ \cite[Proposition 2.2]{bar96}
\begin{equation}\label{eq:bdrydirac}
    A =  \frac{\mc}{2} + c(\n)\dw +\conw_{\n}.
\end{equation} 
In particular, $A$ is an adapted boundary operator in the sense of \cite{bb16}.

\subsection{The boundary identity} The following identity is central to our argument. 

\begin{prop}
    Every harmonic spinor $\psi$ on $W$ with boundary trace $\phi$ on $\Sigma$, satisfies \begin{equation}\label{eq:bdrylich}
    \int_W |\conw \psi|^2+\frac{1}{4}\sw|\psi|^2 = \int_\Sigma  \inner{(A-\mc/2)\phi}{\phi} . 
\end{equation}
\end{prop}
\begin{proof}
    Since $\dw \psi =0$, we use integration by parts on the Lichnerowicz formula for $\dw$. This gives $$\int_W |\conw \psi|^2+\frac{\sw}{4}|\psi|^2 = \int_\Sigma  \inner{\conw_\n \psi}{\psi}. $$
But, since $\dw \psi=0$, we have $\conw_\n \psi = (A-\mc/2)\phi$ along $\Sigma$ by \eqref{eq:bdrydirac}.
\end{proof}

If $\sw > 0$, then the left-hand side of \eqref{eq:bdrylich} is positive for every $\psi \neq 0$. The rest of the section is about the sign of the right-hand side. 

\subsection{The APS boundary condition} Since $A$ is a self-adjoint elliptic operator on the closed manifold $\Sigma$, its eigenspinors form an orthonormal basis of the Hilbert space $L^2(\Sigma, S_W|_{\Sigma})$. We shall denote its negative spectral subspace by $\haps$, i.e., the closed span of eigenspinors with negative eigenvalues. Let $P_{\geq 0}$ denote the projection onto the complement of $\haps$. Then, the harmonic spinor boundary value problem under the Atiyah-Patodi-Singer(APS) boundary condition \cite{aps} is, 
\begin{equation*}
\begin{cases}
    \dw \psi = 0 \qquad \text{in } W \\
    P_{\geq 0}\phi=0 \qquad \text{on } \Sigma.
\end{cases}
\end{equation*}
Here, $\psi \in H^1(W,S_W)$ and its boundary trace $\phi \in H^{1/2}(\Sigma,S_W|_{\Sigma})$. The boundary condition demands that $\phi$ actually lie in $\haps$. This is motivated by the fact that the half-cylinder $\Sigma \times [0,\infty)$ with product metric attached to $W$ by $\Sigma \times \{0\}$ gives, by \eqref{eq:bdrydirac} $\dw = c(\n)(\partial_t - A)$, so an eigenspinor $A\phi = \lambda\phi$ has harmonic extension $e^{\lambda t}\phi$, which is in $L^2$ exactly when $\lambda <0$. 

We shall denote the resulting operator by $\daps$. It is Fredholm, and spinors in its kernel are smooth up to the boundary \cite[Theorem 4.9]{bb16}. The adjoint boundary condition is $P_{>0}\phi=0$, therefore $\daps$ is self adjoint exactly when $\ker A=0$ \cite[Example 5.12]{bb16}. Further, when $n$ is even, $S_W = S_W^+ \oplus S_W^-$. The operator $\dw$ exchanges the chiralities, whereas $A$ and its spectral projections preserve them. We denote by $D^{\pm}_{\mathrm{APS}}$ the restriction of $\daps$ to $S_W^\pm$ and define the APS index by $\indaps(W,g)= \mathrm{Ind} \dapsp$. Then, we have the following proposition. 

\begin{prop}\label{prop:vanishindex}
    Let $n$ be even and $\ker A = \{0\}$. If there is no nonzero harmonic spinor satisfying the APS boundary condition, then $\indaps(W,g) = 0$.
\end{prop}
\begin{proof}
    Since $\ker A = 0$, $\daps$ is self adjoint, so we have $(\dapsp)^*=\dapsm$. Therefore  $\mathrm{Ind} \dapsp =  \dim \ker \dapsp - \dim \ker \dapsm = 0$.
\end{proof}

We shall require two further properties. The first is that $\indaps$ is constant along a continuous path of metrics with invertible boundary Dirac operators. Since $A$ is the Dirac operator of $\Sigma$, its kernel depends only on the boundary metric; therefore, if $\ker A=0$, then $\indaps(W,g)$ depends on $g$ only via $g_{\Sigma}$ and hence can be computed for a $g$ that is a product in a collar of $\Sigma$.  
The second is the gluing formula. If a closed hypersurface $\Sigma$ separates an even-dimensional spin manifold $X$ into $W_1$ and $W_2$, i.e., $X = W_1 \cup_{\Sigma} W_2$, and if the boundary Dirac operators are invertible, then the decomposition theorem \cite[Theorem 1.19]{bb12} gives
\begin{equation}\label{eq:apsgluing}
    \indaps(X,g) = \indaps(W_1,g)+\indaps(W_2,g).
\end{equation}
Note that the two outward normals give boundary operators $A$ and $-A$, and therefore the spectral boundary conditions are complementary. When $X$ is closed, the left-hand side is just the usual index $\ind(X)$ and is equal to $\widehat A(X)$ by the index theorem. In dimension four, this is just a constant times the signature, i.e., $-\sigma(X)/8$.

\subsection{Mean convex boundary} The following corollary is an immediate consequence of equation \eqref{eq:bdrylich} and motivates the next subsection.

\begin{corollary} \label{cor:meanconvex}
If $\sw >0$ and $\mc \geq 0$, then there is no nonzero harmonic spinor satisfying the APS boundary condition. Furthermore, when $n$ is even and $\ker A=\{0\}$, then we have  $\indaps (W,g)=0$.
\end{corollary}
\begin{proof}
  Suppose $\psi$ is a harmonic spinor satisfying the APS boundary condition, i.e., $\phi \in \haps$, then $\int_{\Sigma} \inner{A\phi}{\phi} \leq 0$. Combined with $\mc \geq 0$, this gives that the right-hand side of  \eqref{eq:bdrylich} is non-positive. But, since $\sw>0$, the left-hand side is positive unless $\psi =0$. Hence $\psi=0$ and Proposition \ref{prop:vanishindex} gives $\indaps(W,g)=0$.
\end{proof}

\subsection{Index vanishing from modified stability operator}\label{subsec:vanish}
We now replace the mean convexity condition with a weaker spectral condition $\lambda_1(\mathcal{L}_{\Sigma})>0$. The first step is a lower bound on the first eigenvalue of $A^2-\mc^2/4$.

\begin{lemma}\label{lem:agap}
$\lambda_1(A^2 - \mc^2/4) \geq \frac{1}{2} \lambda_1(\mathcal{L}_{\Sigma})$. In particular, if $\lambda_1(\mathcal{L}_{\Sigma})>0$, then $\ker A = 0$. 
\end{lemma}
\begin{proof}
    By integrating the Lichnerowicz formula for $A$, we have for any test spinor $\phi$, $$\int_{\Sigma}|A\phi|^2  = \int_{\Sigma}|\cons \phi|^2 + \frac{1}{4}\int_{\Sigma} \scs|\phi|^2.$$
    We write $\lambda_1 \coloneqq \lambda_1({\mathcal{L}_{\Sigma}})$. By the variational criterion for $\lambda_1$ we have for all test functions $f$, $$\int_{\Sigma}|\nabla f|^2 + \frac{1}{2}(\scs-\mc^2)f^2 \geq \lambda_1\int_\Sigma f^2.$$
    Now, note that, even though $\phi$ is a section, $|\phi|$ is a well-defined Lipschitz function on $\Sigma$. Further, $|\nabla|\phi|| \leq |\cons \phi|$ by Kato inequality. Therefore, we test against $|\phi|$ and combine the above identity and inequality as follows:
    \begin{align*}
        4 \int_{\Sigma}|A\phi|^2 &= 4  \int_{\Sigma} |\cons \phi|^2 +  \int_{\Sigma} \scs |\phi|^2 \\
                &\geq 4  \int_{\Sigma} |\cons \phi|^2 - 2\int_{\Sigma}|\nabla|\phi||^2 + \int_{\Sigma} (\mc^2 + 2\lambda_1)|\phi|^2 \\
                &\geq 2  \int_{\Sigma} |\cons \phi|^2 + \int_{\Sigma} (\mc^2 + 2\lambda_1)|\phi|^2 \\
                &\geq  \int_{\Sigma} (\mc^2 + 2\lambda_1)|\phi|^2.
    \end{align*}
If $\lambda_1 > 0$, then $A\phi = 0$ forces $\phi = 0$. Further, since this is true for every test spinor $\phi$, dividing by $4$ and rearranging gives the desired conclusion.
\end{proof}

We can now come to the second step, that is, to show that $\lambda_1(A^2-\mc^2/4)>0$ forces a sign on the boundary term in \eqref{eq:bdrylich} for all spinors in $\haps$. 

\begin{prop} \label{prop:rhssign}
    Suppose $\lambda_1(\mathcal{L}_{\Sigma})>0$. Then for every smooth $\phi \in \haps$, we have \begin{equation*}
        \int_{\Sigma} \langle (A-\mc/2)\phi,\phi \rangle \leq 0.
    \end{equation*}
\end{prop}
\begin{proof}
    We first note that it is enough to prove this for finite-dimensional subspaces $E$ of $\haps$ spanned by eigenspinors of $A$. Indeed, suppose $\phi \in \haps$ is smooth, then consider the partial sums $\phi_N$ of the expansion of $\phi$ in terms of eigenspinors of $A$. Then, $\phi_N \to \phi$ and $A\phi_N \to A\phi$ in $L^2$. So both terms on the left-hand side of the inequality converge. 
    Henceforth, we assume that $E \subset \haps$ is such a finite-dimensional subspace and consider the integral above for $\phi \in E$. If it is less than zero for all such $E$ then we are done. Therefore, assume otherwise. In that case the self-adjoint operator $P_E(A- \mc/2)$ on $E$ has an eigenspinor with non-negative eigenvalue\footnote{Here $P_E$ denotes the orthogonal projection onto $E$.}. Since $A$ preserves, this implies that there is a $\phi \in E$ and a $\lambda \geq 0$ such that
    $$(A-\lambda)\phi = \frac{1}{2}P_{E}(\mc\phi).$$
    Taking norms on both sides, and using that $P_E$ does not increase length, we get \begin{align*}
        \frac{1}{4}\int_{\Sigma} \mc^2 |\phi|^2 &\geq \int_{\Sigma} |(A-\lambda)\phi|^2 \\
    &= \int_{\Sigma} |A\phi|^2 + \lambda^2|\phi|^2 - 2\lambda \langle A\phi,\phi\rangle \\
    &\geq \int_{\Sigma} |A\phi|^2.
    \end{align*}
    Note that the last term is non-negative because $\phi \in \haps$ and $\lambda \geq 0$. But, this is a contradiction because by Lemma \ref{lem:agap}, as $\lambda_1 >0$, we have that $$\int_{\Sigma} |A\phi|^2 > \frac{1}{4}\int_{\Sigma} \mc^2 |\phi|^2.  \qedhere $$
\end{proof}

We now prove the main result of this section. 

\begin{theorem}\label{thm:indexwvanish}
    Let $(W,\Sigma,g)$ be a compact spin Riemannian manifold with boundary $\Sigma$. If $\sw >0$ and $\lambda_1(\mathcal{L}_\Sigma)>0$, then there is no nonzero harmonic spinor on $W$ satisfying the APS boundary condition. In particular, if $\dim W$ is even, then $\indaps(W,g) = 0$.
\end{theorem}
\begin{proof}
    Suppose $\psi$ is a harmonic spinor satisfying the APS condition, i.e., its boundary trace $\phi \in \haps$. Therefore, it satisfies \eqref{eq:bdrylich}. Since, $\sw >0$, the left hand side is positive, unless $\psi=0$, and since $\lambda_1(\mathcal{L}_\Sigma)>0$, by Proposition \ref{prop:rhssign}, the right hand side is non-positive. Therefore, $\psi = 0$. Further, $\ker A=0$ by Lemma \ref{lem:agap}. When $\dim W$ is even, Proposition \ref{prop:vanishindex} implies $\indaps(W,g) = 0$.
\end{proof}

\section{Spin fillings with vanishing index}\label{sec:filling}
We now prove that the index of spin fillings vanishes, without any curvature condition in the interior, as long as the filling lies in a four-manifold with vanishing rational second homology. The following is the main result of this section. 

\begin{theorem}\label{thm:fillingindex}
   Let $U$ be a compact spin four-manifold with $H_2(U;\QQ)=0$, and let $W$ be a compact domain in the interior of $U$, with a smooth boundary $(\Sigma,h)$ with $\sh >0$, and equip $W$ with the spin structure induced from $U$. Then, for every Riemannian metric $g$ on $W$ extending $h$, we have $\ker A=0$ and $\indaps (W,g)=0$. \end{theorem}

The proof is at the end of this section. The rest of the section is devoted to developing tools for proving this. We first make the following easy observation. 

\begin{lemma}\label{lem:spinindexmetric}
    Let $W^4$ be a compact spin manifold with a smooth boundary $(\Sigma,h)$ with $\sh >0$, and let $g$ be a metric on $W$ extending $h$. Then, the condition that $\ker A=0$ and the value of $\indaps (W,g)$ are independent of both $g$ and the psc metric $h$.
\end{lemma}
\begin{proof}
    The Lichnerowicz formula and $\sh >0$ give $\ker A = 0$. By Bamler-Kleiner \cite{bk19}, the space of psc metrics on any closed three-manifold is path connected. Attaching the trace of such a path to $W$ proves the result, because the boundary Dirac operators along the path remain invertible. 
\end{proof}

Lemma \ref{lem:spinindexmetric} allows us to vary metrics for a fixed filling. However, we also need to keep track of how the index changes with the filling. To this end, we make the following definition.

\begin{definition}(Boundary Invariant)
Let $(\Sigma,h)$ be a closed spin three-manifold with $\sh >0$. Let $W$ be any spin filling of $\Sigma$ inducing the given spin structure on $\Sigma$, with a metric $g$ extending $h$. We define the boundary invariant of $\Sigma$ as \begin{equation}
    \nu(\Sigma) \coloneqq \indaps(W,g) + \frac{1}{8}\sigma(W).
\end{equation}
\end{definition}

\begin{lemma}\label{lem:bdryinv}
    The boundary invariant $\nu(\Sigma)$ is independent of the psc metric $h$, as well as of the filling $(W,g)$. 
\end{lemma}
\begin{proof}
    Suppose we have two fillings $(W_1,g_1)$ and $(W_2,g_2)$ for psc metrics $h_1$ and $h_2$ respectively. We can add a neck joining $h_1$ and $h_2$; by Lemma \ref{lem:spinindexmetric}, this does not change the index. We then glue the two fillings on two sides of the neck to form a closed spin manifold $V = W_1 \cup_{\Sigma} \overline{W}_2$; here $\overline{W}_2$ denotes $W_2$ with orientation reversed. Since $\indaps(\overline{W}_2,g_2)=-\indaps(W_2,g_2)$, by the index gluing formula and additivity of signature we have $$\indaps(W_1,g_1) - \indaps(W_2,g_2) = \widehat A(V) = -(\sigma(W_1)-\sigma(W_2))/8. $$
    Rearranging gives the desired conclusion. 
\end{proof}
\begin{corollary} \label{cor:nunotchange}
    Let $(V,g_V)$ be a spin cobordism from $(\Sigma_1,h_1)$ to $(\Sigma_2,h_2)$ with $g_V$ a product near both ends. If $\scal_{V} >0$, then \begin{equation}
        \nu(\Sigma_2)-\nu(\Sigma_1)= \sigma(V)/8.
    \end{equation}
\end{corollary}
\begin{proof}
    By Lemma \ref{lem:bdryinv}, one can compute $\nu(\Sigma_2)$ using $W \cup_{\Sigma_1}V$ where $W$ is a spin filling of $(\Sigma_1,h_1)$ with a product collar. This gives, by the gluing formula and additivity of signature, $$\nu(\Sigma_2)-\nu(\Sigma_1)= \indaps(V,g_V)+ \sigma(V)/8.$$
    The index term vanishes by Corollary \ref{cor:meanconvex} because the product collars make the boundary totally geodesic, hence of mean curvature $0$, and $\ker A=0$ since $h_1$ and $h_2$ are psc. 
\end{proof}

\subsection{Surgery argument} In this subsection we will use surgery to simplify $(W,\Sigma)$. Since $\Sigma$ admits a psc metric, by Perelman's resolution of the geometrisation conjecture, each component of $\Sigma$ is a connected sum of spherical space forms and copies of $S^2 \times S^1$ ; cf. \cite[\S 1]{mar12}. A connected sum of spherical space forms is a rational homology sphere. We shall first remove the $S^2 \times S^1$ pieces and join the components to go from $\Sigma$ to a rational homology sphere $Y$ without changing the boundary invariant $\nu$. Then, we will use $H_2(U;\QQ)=0$, to construct a rational homology ball $Z$ filling $Y$.

\begin{lemma} \label{lem:goingtoy}
    Let $\Sigma$ be an orientable closed three-manifold admitting a psc metric. Suppose a spin structure is fixed on $\Sigma$; then there is a spin cobordism $V$ from $\Sigma$ to a rational homology sphere  $Y$ such that $V$ admits a psc metric that is a product near both ends and $\sigma(V)=0$.
\end{lemma}
\begin{proof}
    Since $\Sigma = Y \#_k(S^2 \times S^1)$, we attach $3$-handles to $\Sigma \times [0,1]$ along each $S^2 \times \{pt\}$ generating $H_2(\Sigma;\QQ)$. This gives the required cobordism as the spin structure extends over the handles. 
    Now flip the cobordism to view it as a cobordism from $Y$ to $\Sigma$; it has only $1$-handles. Therefore since $Y$ is a rational homology sphere, $H_2(V;\QQ)=0$ and $\sigma(V)=0$. Starting from a psc metric on $Y$, since the cobordism to $\Sigma$ is a Gromov-Lawson cobordism ($1$-handles are surgeries of codimension $3$), its trace admits a psc metric as well, which can be made a product near the ends.  If $\Sigma$ is disconnected we do this for each component and then join each of the rational homology spheres by $1$-handles. This doesn't change $H_2$ and just like above the psc metric extends over $V$.
\end{proof}
\begin{corollary}
    $\nu(\Sigma) = \nu(Y)$.
\end{corollary}
\begin{proof}
    Follows from Corollary \ref{cor:nunotchange} as $\sigma(V)=0$.
\end{proof}

We now construct a rational homology ball filling $Y$. 

\begin{lemma}\label{lem:boundsball}
 Let $U$ be as in Theorem \ref{thm:fillingindex} and let $W$ be a compact domain in the interior of $U$, with a smooth boundary $(\Sigma,h)$ that admits a psc metric. Then, for $V$ and $Y$ as in Lemma \ref{lem:goingtoy}, $Y$ bounds a spin rational homology ball with spin structure induced by $V$.
\end{lemma}
\begin{proof}
   Observe that $U = W \cup_{\Sigma}W^c$. Since $H_2(U;\QQ)=0$, by Mayer-Vietoris, we get a surjection $H_2(\Sigma;\QQ) \to H_2(W;\QQ)$. Therefore, $H_2(W;\QQ)$ is actually spanned by the images of the spheres $S^2 \times \{pt\}$ that span $H_2(\Sigma;\QQ)$. So attaching the $3$-handles that created $V$ gives $Z_1= W \cup_{\Sigma}V$ with $b_2(Z_1)=0$ and boundary $Y$. 
    We now kill the $b_1$. To this end, we choose disjoint interior circles which form a basis of $H_1(Z_1;\QQ)$. We perform $1$-surgery along each of these such that the chosen framings extend the spin structure. This takes $b_1$ to zero. However, each such surgery step increases the Euler characteristic by $2$ because it replaces $S^1 \times D^3$  by $D^2 \times S^2$. Since the boundary $Y$ is a rational homology sphere and is not affected, the Poincar\'e-Lefschetz duality and the exact sequence of pairs give $b_3(Z)=b_1(Z)$ at each step. Therefore $\chi = 1-2b_1+b_2$ and hence $b_2$ remains unchanged. Therefore, the resulting spin manifold $Z$ has boundary $Y$ and $b_1(Z)=b_2(Z)=b_3(Z)=0$.
\end{proof}

\subsection{Index vanishing} We will use the following proposition due to Baraglia \cite[Proposition 3.4]{bara} in our spin case. 
\begin{prop}[Baraglia]
\label{prop_Froyshov-Baraglia}
    Let $Z$ be a compact spin four-manifold with $b_1=b_2=0$. Suppose its boundary $(Y,h)$ is a rational homology sphere with $\sh>0$. Suppose the metric on $Z$ has a product collar near $Y$. Then we have $\indaps (Z,g)=0$.  
\end{prop}
    Baraglia proved more general results than Proposition \ref{prop_Froyshov-Baraglia}, and the arguments in \cite{bara} are more involved than is necessary for our case. We present a direct proof of Proposition \ref{prop_Froyshov-Baraglia} here using only the techniques developed in \cite{kronheimer2007monopoles}.
\begin{proof}
 We consider a slightly extended case (which is also a special case of Baraglia's result), where $Z$ is endowed with a spin$^c$ structure $\mathfrak{t}$, and the spin$^c$ structure $\mathfrak{t}$ on $Z$ induces a spin$^c$ structure $\mathfrak{s}$ on $Y$. Let $A_0$ be a spin$^c$ connection on $Y$ that induces a flat connection on the determinant line bundle, and let $A$ be a spin$^c$ connection on $Z$ that is equal to the pull-back of $A_0$ from $Y$ on the product collar. Let $D_A$ denote the Dirac operator associated with $A$. We show that 
 \begin{equation}
 \label{eqn_indaps_DA=0}
 \indaps D_A=0.
 \end{equation}

    If $Y=S^3$, $Z=D^4$, then the desired result follows from a straightforward calculation using the APS index theorem. We will show that $\indaps D_A$ does not depend on the choices of $Y$, $\mathfrak{s}$, $Z$, $\mathfrak{t}$, $A_0$, $A$. In the following, let $\mathfrak{s}_0$ denote the unique spin$^c$ structure on $S^3$.

    Recall that \cite{kronheimer2007monopoles} introduced three versions of monopole Floer homology for a closed oriented $3$-manifold $Y$ with a spin$^c$ structure $\mathfrak{s}$, which are denoted by 
    $\widecheck{\mathit{HM}}_{*}(Y,\mathfrak{s}), \widehat{\mathit{HM}}_{*}(Y,\mathfrak{s}),\overline{\mathit{HM}}_{*}(Y,\mathfrak{s})$. In the notation of \cite{kronheimer2007monopoles}, these are $\mathbf{Z}[U_\dagger]$ modules, where the action of $U_\dagger$ has degree $-2$.
    
    If $Y$ is a rational homology sphere, then there is exactly one reducible solution to the Seiberg-Witten equation on $(Y,\mathfrak{s})$, and hence we have 
    \[
    \overline{\mathit{HM}}_{*}(Y,\mathfrak{s})\cong \overline{\mathit{HM}}_{*}(S^3,\mathfrak{s}_0)\cong \mathbf{Z}[U_\dagger,U_\dagger^{-1}]
    \]
    as $\mathbf{Z}[U_\dagger]$-modules.
    If we further assume that $Y$ has positive scalar curvature, then there are no irreducible flow lines under sufficiently small perturbations, which implies that
    \[
    \widehat{\mathit{HM}}_{*}(Y,\mathfrak{s})\cong \widehat{\mathit{HM}}_{*}(S^3,\mathfrak{s}_0)\cong \mathbf{Z}[U_\dagger],
    \]
    \[
    \widecheck{\mathit{HM}}_{*}(Y,\mathfrak{s})\cong \widecheck{\mathit{HM}}_{*}(S^3,\mathfrak{s}_0)\cong \mathbf{Z}[U_\dagger,U_\dagger^{-1}]/\mathbf{Z}[U_\dagger].
    \]
    as $\mathbf{Z}[U]$-modules. We refer the reader to \cite[Proposition 36.1.3]{kronheimer2007monopoles} for more details.

    For a general $3$-manifold $Y$, there is an exact triangle (\cite[Proposition 22.2.1]{kronheimer2007monopoles})
    \begin{equation}
        \label{eqn_exact_sequence_HM}
\cdots \xrightarrow{j_*} \widehat{\mathit{HM}}_{*}(Y,\mathfrak{s}) \xrightarrow{p_*}  \overline{\mathit{HM}}_{*}(Y,\mathfrak{s}) \xrightarrow{i_*}  \widecheck{\mathit{HM}}_{*}(Y,\mathfrak{s}) \xrightarrow{j_*} \widehat{\mathit{HM}}_{*}(Y,\mathfrak{s}) \xrightarrow{p_*} \cdots 
    \end{equation}
    There is also a canonical rational grading for the monopole Floer homology of $(Y,\mathfrak{s})$ when $c_1(\mathfrak{s})$ is torsion (\cite[Section 28.3]{kronheimer2007monopoles}). 
    
    If $Y$ is a rational homology sphere, then the \emph{Fr\o yshov invariant} $h(Y,\mathfrak{s})$ is defined to be the rational number such that 
    \begin{equation}\label{defn_Froyshov}
    -2h(Y,\mathfrak{s})-2
    \end{equation}
    is equal to the highest degree in the image of $p_*$ in \eqref{eqn_exact_sequence_HM}. The normalization coefficients in \eqref{defn_Froyshov} are chosen so that $h(S^3,\mathfrak{s}_0)=0$ and that as $Y$ ranges over all integer homology spheres, the invariant $h(Y,\mathfrak{s})$ takes every integer value.

    If $(Y_1,\mathfrak{s}_1)$ and $(Y_2,\mathfrak{s}_2)$ are rational homology spheres with spin$^c$ structures, and if there exists a spin$^c$ cobordism $(W,\mathfrak{t})$ from $Y_1$ to $Y_2$ such that $b_1(W)=b_2(W)=0$ (such a cobordism is called a \emph{rational homology cobordism}), then it is known that $h(Y,\mathfrak{s}_1)=h(Y_2,\mathfrak{s}_2)$. This can be seen from the functoriality of monopole Floer homology; we sketch the argument here for the reader's convenience: The cobordism $(W,\mathfrak{t})$ induces homomorphisms from the monopole Floer homology of  $(Y_1,\mathfrak{s}_1)$ to the monopole Floer homology of $(Y_2,\mathfrak{s}_2)$. If $W$ is a rational homology cobordism, then the induced homomorphisms are homogeneous with degree zero, and a direct computation of reducible solutions shows that it induces an isomorphism on $\overline{\mathit{HM}}_{*}$. The homomorphisms are also commutative with the maps in \eqref{eqn_exact_sequence_HM}. By a straightforward algebraic argument, this implies $h(Y_1,\mathfrak{s}_1)\ge h(Y_2,\mathfrak{s}_2)$. Reversing the orientation of $W$ yields the reversed inequality.

    Now let $Y$, $Z$ be given as in the statement of Proposition \ref{prop_Froyshov-Baraglia}, and assume that $\mathfrak{s}$ is a spin$^c$ structure on $Y$, and $\mathfrak{t}$ is a spin$^c$ structure on $Z$ that induces the boundary spin$^c$ structure $\mathfrak{s}$. Since every spin structure defines a spin$^c$ structure, this setting is more general than the original setup of Proposition \ref{prop_Froyshov-Baraglia}.

    Let $W$ be obtained by removing an interior open ball from $Z$. Then $W$ is a rational homology cobordism from $S^3$ to $Y$. Therefore, 
    \[
    h(Y,\mathfrak{s}) = h(S^3,\mathfrak{s}_0) = 0.
    \]
    
    Since $Y$ is a rational homology sphere with a psc metric, the elements with the highest degree in the image of $p_*$ have the form $k[\mathfrak{a}]$, where $k$ is a non-zero integer and $\mathfrak{a}$ is the boundary-unstable solution with the highest degree. The absolute grading of $\mathfrak{a}$ is defined by \cite[Definition 28.3.1]{kronheimer2007monopoles}, and it has the form 
    \[
    \deg \mathfrak{a} = x_1 + \dim^{\text{vir}} \mathcal{M}(Z,\mathfrak{a}),
    \]
    where $x_1\in \mathbf{Q}$ is a constant that only depends on the cohomology of $(Z,Y)$, and $\mathcal{M}(Z,\mathfrak{a})$ denotes the moduli space of solutions to the Seiberg-Witten equations on $Z\cup Y\times [0,+\infty)$ that converge to $\mathfrak{a}$ on the end. The virtual dimension formula of $\mathcal{M}(Z,\mathfrak{a})$ has the form 
    \begin{equation}
    \label{eqn_vir_dim_formula}
\dim^{\text{vir}} \mathcal{M}(Z,\mathfrak{a}) = x_2 + 2\,\indaps D_A,
    \end{equation}
    where $x_2\in \mathbf{Q}$ is a constant that only depends on the cohomology of $(Z,Y)$, and $A$ is an arbitrary spin$^c$ connection on $(Z,\mathfrak{t})$ that agrees with $\mathfrak{a}$ on a collar neighborhood of $Y$. The factor of $2$ on the APS index in \eqref{eqn_vir_dim_formula} comes from the fact that the APS index is defined over $\mathbf{C}$, while the virtual dimension is defined over $\mathbf{R}$. 

    Therefore, we have
    \[
    2\,\indaps D_A + (x_1 + x_2) = \deg \mathfrak{a} = -2 h(Y,\mathfrak{s}) - 2 = -2.
    \]
    As a consequence, $\indaps D_A$ does not depend on the choices of $Y$, $\mathfrak{s}$, $Z$, $\mathfrak{t}$, $A_0$, $A$, as long as the cohomology of $(Z,Y)$ satisfies the above conditions, and $A_0$ is a spin$^c$ connection on $Y$ that induces a flat connection on the determinant line bundle, and $A$ is a spin$^c$ connection on $Z$ that is equal to the pull-back of $A_0$ from $Y$ on the product collar. Note that if $\mathfrak{t}$ is given by a spin structure on $Z$ and $A$ is given by the spin connection, then the condition on the connection is satisfied. Since \eqref{eqn_indaps_DA=0} holds when $(Z,Y)=(D^4,S^3)$, we conclude that the same equation holds for arbitrary $Y$, $\mathfrak{s}$, $Z$, $\mathfrak{t}$, $A_0$, $A$ satisfying the above assumptions. 
\end{proof}

\begin{corollary}\label{cor:nuvanishes}
   $\nu(Y)=0$.
\end{corollary}
\begin{proof}
    This reduces to $\nu(Y)=\sigma(Z)/8$, but that is zero because $Z$ is a rational homology ball.
\end{proof}
\begin{proof}[Proof of Theorem \ref{thm:fillingindex}]
    By Lemma \ref{lem:spinindexmetric}, $\ker A=0$. Since $W$ lies in $U$ and $H_2(U;\QQ)=0$, its intersection pairing vanishes and therefore $\sigma(W)=0$. Thus, $$\indaps(W,g)=\nu(\Sigma)=\nu(Y)=0.$$
The second equality is by Corollary \ref{cor:nunotchange} and the third is from Corollary \ref{cor:nuvanishes}.  
\end{proof}

\section{Proof of Theorem \ref{maintheoremB} } \label{sec:final}

\begin{proof}[Proof of Theorem \ref{maintheoremB}]
    Let $S$ be a finite graph and $U$ be a regular neighbourhood of $S$. We have $H_2(U;\QQ)=0$ because $U$ deformation retracts onto $S$. Let $g$ be a complete metric on $\xs$ with $\scal \geq \kappa^2$. Then Proposition $\ref{prop:spectralgap}$ gives a hypersurface $\Sigma \subset U\backslash S$  with $\lambda_1(\mathcal{L}_{\Sigma})>0$. We write $(W,\Sigma,g)$ for the compact part with boundary $\Sigma$. The closure of its complement is a compact domain $(C,\Sigma,\hat g)$ lying inside of $U$ and containing $S$. Here, $\hat g$ is some smooth metric on $X$ that agrees with $g$ on $W$ and on a collar of $\Sigma$ in $C$. By Lemma \ref{lem:agap}, the boundary Dirac operator of $\Sigma$ is invertible. By the gluing formula \eqref{eq:apsgluing} along $\Sigma$ and the index theorem for the closed manifold $X$, we have $$\widehat A(X) = \indaps(W,g)+\indaps(C,\hat g).$$
We now show that both terms on the right side are zero. First, since, $\lambda_1(\mathcal{L}_{\Sigma})>0$ by Theorem \ref{thm:indexwvanish}, $\indaps(W,g) =0.$ Again, since  $\lambda_1(\mathcal{L}_{\Sigma})>0$, $g_{\Sigma}$ is conformal to a psc metric $h$. Along a conformal deformation, the boundary Dirac operator stays invertible, so the index doesn't change. Since $h$ is psc, and $(C,\Sigma)$ lies in $U$, by Theorem \ref{thm:fillingindex} $\indaps(C,\hat g)=0.$ Combining both gives that $X$ has $\widehat A(X) = 0$. Contradiction. 
    
\end{proof}

\section{The conformal argument}\label{sec:six}
The following is the main result of this section and contains the essential part of the proof of Theorem \ref{thm:singularthm}. 
\begin{theorem}\label{thm:conf}
    Let $X$ be a closed four-manifold and $S \subset X$ is a finite graph. If $X$ admits an $L^\infty$-metric $g$ that is smooth on $X \backslash S$ with $\scal_g \geq 0$, then \begin{enumerate}
        \item If $\ric_g \nequiv 0$ on $X \backslash S$, then there is a complete metric on $X \backslash S$ with uniformly positive scalar curvature. 
        \item If $\ric_g \equiv 0$ on $X\backslash S$, then $g$ can be extended to a smooth Ricci flat metric on $X$ for a smooth structure that agrees with the original one on $X \backslash S$. 
    \end{enumerate}
\end{theorem}
The proof of (1) is essentially contained in \cite{lm19} where they show that existence of a complete $\scal >0$ metric. Our main observation is just that the conformal factor may be chosen to have a uniformly positive lower bound on the scalar curvature of the deformed metric. Before, we prove the theorem, we look at the model case for the conformal argument that provides the model for the desired end. Recall that in dimension four, given a function $u>0$, the Riemannian metric $g$ may be conformally deformed as $\bar g \coloneqq u^2 g$ and the scalar curvature transforms as 
\begin{equation*}
\sgg = u^{-3}(-6\Delta_g u + \sg\,  u).
\end{equation*}
Therefore, to get a uniformly positive lower bound one needs to construct a conformal factor $u$ such that on the end \begin{equation}\label{eq:conformalsc} 
    -6\Delta_g u + \sg\,  u \geq \kappa^2 u^3.
\end{equation}

\subsection{Model for the end} \label{subsec:model} We first consider $\RR^4$ with the flat metric $\geu$. Since our singularities are at most one dimensional so we need to consider the case of isolated points and lines respectively. 

First, let $r$ denote distance from the origin. Taking $u=r^{-1}$, since $-\Delta u=u^3$, we get a round cylinder $S^3 \times \RR$ with $\bar g= r^{-2}\geu$, $$\bar g = g_{S^3}+ \frac{dr^2}{r^2}, \qquad \sgg = 6.$$
Next, consider a line and let $r$ denote the distance to it. Taking $u=r^{-1}$, since $u$ is harmonic, we get $S^2 \times \HH^2$ with $\bar g = r^{-2}\geu$, $$\bar g = g_{S^2}+ \frac{ds^2+dr^2}{r^2}, \qquad \sgg =0. $$ Since, we get a scalar flat metric, this is a border line case, and still in this case we can get a uniformly positive lower bound by a logarithmic modification. Specifically, consider $u = (r\sqrt{-\log r})^{-1}$ for $r<1$. Then, with $\bar g = u^2 \geu$, a direct calculation gives us as $r \to 0$ $$ \sgg= 3-\frac{9}{2(-\log r)} \to 3. $$
Note, that we do not get a positive scalar curvature on the entirety of the space only on the end, and that is enough for for purpose. Note that $\bar g$ is complete on the end towards $r=0$. Indeed, we have
\begin{equation}\label{eq:completeness}
\int_0^{\delta} u \,dr = \int_0^{\delta} \frac{dr}{r(-\log r)^{1/2}} = \infty      
\end{equation}

\subsection{Proof of Theorem \ref{thm:conf}} Before, we begin the proof, we setup some notation. Let $X,S$ and $g$ be as in Theorem \ref{thm:conf}. We shall denote by $r$ the distance from $S$, and we shall consider a $\delta >0$ small enough so that $\{0<r \leq \delta\}$ is a tubular neighbourhood of $S$ in $X$. We take $w$ to be a function on $\{0 <r \leq \delta\}$ such that $w=r^{-1}$ near isolated points in $S$ and $w= (r\sqrt{-\log r})^{-1}$ near the one dimensional components. 

To prove part (1) of the theorem, we need to find a conformal factor $u$ on $X\backslash S$ whose ends are modelled on $\S$\ref{subsec:model}, i.e., for some $C \geq 1$, \begin{equation} \label{eq:bounds}
    C^{-1} w \leq u \leq C w \qquad \text{on } \{0 < r \leq \delta\}.
\end{equation}
This will guarantee completeness of $\bar g$ by \eqref{eq:completeness}.

\begin{proof}[Proof of Theorem \ref{thm:conf}] We first prove (1). We first claim that we can assume that $\sg>0$ $X\backslash S$. Indeed, either $\sg \nequiv 0$, in which case it follows from \cite[Corollary 4.2]{lm19}, or $\sg \equiv 0$, but since $\ric_g \nequiv 0$, the same holds after perturbing $g$ in the direction $-\ric_g$ as in the proof of \cite[Theorem 1.7, pg. 122-124]{lm19}.\footnote{Note that the there it is stated for $n=3$, but the proof works for any $n$.} 
Therefore, it is enough to find $u \in C^{\infty}(X\backslash S)$ satisfying by \eqref{eq:conformalsc} and \eqref{eq:bounds}. Let $f$ be a smooth positive function on $\xs$ such that \begin{equation} \label{eq:boundsonf}
C^{-1}w^3 \leq f \leq Cw^3 \qquad \text{on } \{0<r\leq \delta\},    
\end{equation}
and further, take $\lscal$ be a smooth function on $\xs$ such that $0<\lscal\leq \min\{1,\sg\}$. As in the proof of \cite[Proposition 6.2, pg. 119]{lm19} with $f$ instead of a Dirac delta function on $S$, let $u$ be a distributional solution of $$-6\Delta_g u + \lscal \, u = f \qquad \text{on } X.$$
Such an $u$ exists because $f $ is integrable(because even near the end $w\in L^3$; cf. \cite[Theorem 5.1]{sy88}). Then, $u$ is positive and smooth on $\xs$, and the Greens function estimate in \cite[proof of Proposition 6.2]{lm19} converts the bounds on $f$ \eqref{eq:boundsonf} into the desired bounds on $u$ \eqref{eq:bounds} after possibly taking a bigger $C$ and a smaller $\delta$.  Since, $\sg \geq \lscal$, \eqref{eq:bounds} and \eqref{eq:boundsonf} gives us $$-6\Delta_g u + \sg \, u = f + (\sg-\lscal)u \geq f \geq C^{-4}u^3 \qquad \text{on } \{0<r \leq \delta\},  $$
and also on $\{r\geq\delta\}$ because that is a compact subset of $\xs$ and therefore, $f$ is bounded there by a positive number and $u$ is bounded above. Therefore, \eqref{eq:conformalsc} holds for some $\kappa^2$.

We now prove (2). Since a finite graph has codimension $3$, by \cite[Theorem 1.2]{as26}, the metric $g$ extends to a smooth Ricci flat metric on $X$ for a smooth structure which agrees with the original one on $\xs$ because in the proof of their theorem they construct it using harmonic coordinates for $g$.
\end{proof}
\begin{remark}
    We note that the proof of Theorem \ref{thm:conf}(1) works in all dimensions $n \geq 3$ with finite subcomplex $S$ such that $\dim S \leq \lfloor\frac{n-2}{2} \rfloor$, with $\bar g = u^{4/(n-2)}g$. In each critical case, like in dimension four above, the logarithmic correction is needed. However, for Theorem \ref{thm:conf}(1), the use of \cite[Theorem 1.2]{as26} makes it specific to dimension four. 
\end{remark}

\section{Proof of Theorems \ref{thm:singularthm} and \ref{thm:rigidity}} \label{sec:seven}
\begin{proof}[Proof of Theorem \ref{thm:singularthm}]
Let $X,S$ and $g$ be as in Theorem \ref{thm:singularthm}. If $\ric_g \nequiv 0$ on $\xs$, then Theorem \ref{thm:conf}(1) gives a complete metric on $\xs$ with uniformly positive scalar curvature contradicting the hypothesis. So $\ric_g = 0$ on $\xs$ and $g$ extends by Theorem \ref{thm:conf}(2).    
\end{proof}
\begin{proof}[Proof of Theorem \ref{thm:rigidity}]
By Theorem \ref{maintheoremB}, $\xs$ doesn't admits a complete metric with uniformly positive scalar curvature, and therefore hypothesis of Theorem \ref{thm:singularthm} are satisfied and $g$ extends to a Ricci flat metric  on $X$ with a smooth structure that agrees with the original one on $\xs$. We write $\overline X$ for $X$ with this smooth structure. Like $X$, $\overline X$ is spin and has non-zero signature so it admits a non-zero harmonic spinor but it also admits a Ricci flat metric, so by Lichnerowicz formula the spinor is parallel. Therefore, the metric is hyperkähler and $\overline{X}$ is diffeomorphic to $K3$ by \cite[Theorem 1]{hit74}.
\end{proof}

\bibliography{bib}
\bibliographystyle{alpha}
\end{document}